\documentclass[12pt]{amsart}
\usepackage{txfonts}
\usepackage{amssymb}
\usepackage{eucal}
\usepackage{graphicx}
\usepackage{amsmath}
\usepackage{amscd}
\usepackage[all]{xy}           

\usepackage{amsfonts,latexsym}
\usepackage{xspace}
\usepackage{epsfig}
\usepackage{float}
\usepackage{color}
\usepackage{fancybox}
\usepackage{colordvi}
\usepackage{multicol}
\usepackage{colordvi}
\usepackage{stmaryrd}

\newtheorem{theorem}{Theorem}[section]

\newtheorem{lemma}[theorem]{Lemma}

\newtheorem{prop-def}{Proposition-Definition}[section]

\newtheorem{remark}[theorem]{Remark}

\newcommand{\nc}{\newcommand}
\newcommand{\delete}[1]{}

\nc{\mlabel}[1]{\label{#1}}  
\nc{\mcite}[1]{\cite{#1}}  
\nc{\mref}[1]{\ref{#1}}  
\nc{\mbibitem}[1]{\bibitem{#1}} 

\delete{
\nc{\mlabel}[1]{\label{#1}  
{\hfill \hspace{1cm}{\bf{{\ }\hfill(#1)}}}}
\nc{\mcite}[1]{\cite{#1}{{\bf{{\ }(#1)}}}}  
\nc{\mref}[1]{\ref{#1}{{\bf{{\ }(#1)}}}}  
\nc{\mbibitem}[1]{\bibitem[\bf #1]{#1}} 
}

\nc{\bfk}{\mathbf{k}}
\nc{\Der}{\mathrm{Der}}
\nc{\Ker}{\mathrm{Ker}}

\begin{document}

\title{ 3-Lie algebra $A_{\omega}^{\delta}$-modules and induced modules }\footnotetext{ Corresponding author: Ruipu Bai, E-mail: bairuipu@hbu.edu.cn.}

\author{RuiPu  Bai}
\address{College of Mathematics and Information  Science,
Hebei University \\
Key Laboratory of Machine Learning and Computational\\ Intelligence of Hebei Province, Baoding 071002, P.R. China} \email{bairuipu@hbu.edu.cn}

\author{Yue Ma}
\address{College of Mathematics and Information  Science,
Hebei University, Baoding 071002, China} \email{mayue3670@163.com}

\author{Pei Liu}
\address{College of Mathematics and Information  Science,
Hebei University, Baoding 071002, China} \email{liupeimath@163.com}

\date{}

\begin{abstract}
In this paper, we define the induced modules of  Lie algebra ad$(B)$ associated with a 3-Lie algebra $B$-module,  and study   the relation between 3-Lie algebra $A_{\omega}^{\delta}$-modules and induced modules of inner derivation algebra ad$(A_{\omega}^{\delta})$.
We construct two infinite dimensional  intermediate series modules   of 3-Lie algebra $A_{\omega}^{\delta}$, and two  infinite dimensional   modules $(V, \psi_{\lambda\mu})$  and $(V, \phi_{\mu})$ of the Lie algebra
  ad$(A_{\omega}^{\delta})$, and prove that only  $(V, \psi_{\lambda0})$ and  $(V, \psi_{\lambda1})$ are induced modules.
\end{abstract}

\subjclass[2010]{17B05, 17B30}
\keywords{ 3-Lie algebra,  3-Lie algebra-module, induced module, intermediate series module}
\maketitle
\vspace{-.5cm}


\numberwithin{equation}{section}

\allowdisplaybreaks

\section{Introduction}
n-Lie algebras were introduced in 1985 (\cite{Filippov}), and their structures were studied extensively \cite{Bai1,Bai2,Bai3,Farrill,Izquier1,Izquier2,Kasymov,Ling,Pozhidaev1,Pozhidaev2,Shengyunhe1}.  Later, 3-Lie algebras were applied in  noncommutative geometry and the quantum geometry of branes in M-theory
 (\cite{Bagger1,Bagger2,DebeSR,Gustavsson}).

 A canonical Nambu 3-Lie algebra (\cite{Nambu}) is a triple of classical observables on a three-dimensional phase space with coordinates $x, y, z,$
\begin{equation}\label{eq:nambu}
    [f_1, f_2, f_3]_\partial = \frac {\partial(f_1, f_2, f_3)}{\partial(x, y, z)},
\end{equation}
 which
is related to Nambu mechanics (generalizing Hamiltonian mechanics by using more than
one hamiltonian). The algebraic formulation of this theory is due to Takhtajan (\cite{T}), see also \cite{DebeSR,Gautheron}. Through years of work,   3-Lie algebras can be easily obtained from Lie algebras, commutative algebras, and Pre-Lie algebras (\cite{Aw,Bai4,Bai5,Bai6}).

The representation theory of
n-Lie algebras was first introduced in \cite{Kasymov}. The adjoint representation $(A,$ ad) of any n-Lie algebra $A$ is defined by
ad$: A\wedge A\rightarrow gl(A)$, for $\forall x, y\in A$, ad$(x, y): A\rightarrow A$,  $\forall z\in A$, $ad(x, y)z$ is the ternary bracket. In \cite{Bai6}, the infinite-dimensional simple 3-Lie algebra  $A_{\omega}^{\delta}$ was introduced; its adjoint representation was investigated there as well. It is proved that the regular representation of $A_{\omega}^{\delta}$ is a Harish-Chandra module, and the natural module of the inner derivation algebra is an intermediate series module.
In \cite{Shengyunhe2}, authors provided a new approach to representation theory of 3-Lie algebras,
which is a generalized representation of a 3-Lie algebra. They also gave examples of generalized representations of 3-Lie algebras, and computed 2-cocycles of the new cohomology.

In this paper, we continue to study the representations of 3-Lie algebras. First we construct two  infinite dimensional 3-Lie algebra  $A_{\omega}^{\delta}$-modules $T_{\lambda0}$ and $T_{\lambda1}$. Then we study the induced module of the Lie algebra ad$(B)$  associated with a 3-Lie algebra $B$ module.  We also construct two  modules $(V, \psi_{\lambda\mu}), (V, \phi_{\mu})$  of inner derivation algebra ad$(A_{\omega}^{\delta})$, and we prove  that $(V, \psi_{\lambda0})$ and $(V, \psi_{\lambda1})$ are induced modules, others are not.

Unless otherwise stated,  algebras and vector spaces are over a field $\mathbb F$ of characteristic zero,  and  $ Z$ is the set of integers.

\section{Preliminaries}

In this section we collect some basic  notions of $3$-Lie algebras.

 {\it A 3-Lie algebra} \cite{Filippov} is a vector space $B$  with a linear multiplication (or a $3$-Lie bracket) $[~,~ ,~ ]: B \wedge B\wedge B\rightarrow B$ satisfying, $\forall x_1, x_2, x_3, x_4, x_5\in B,$

\begin{equation}\label{eq:jacobi}
[x_1, x_2, [x_3, x_4, x_5]]=[[x_1, x_2, x_3], x_4, x_5] +[x_3, [x_1, x_2, x_4], x_5] +[x_3, x_4, [x_1, x_2, x_5]].
 \end{equation}

For $x_1, x_2\in B,$ the linear mapping $\mbox{ad}(x_1, x_{2}):
B\rightarrow B$,
\begin{equation}\label{eq:left}
\mbox{ad}(x_1, x_{2})(x)=[x_1, x_2, x], ~ \mbox{ for all} ~  x\in B,
\end{equation}
is called {\it a left multiplication} defined by
$x_1$, $x_{2} \in B$.

Thanks to \eqref{eq:jacobi},  left multiplications and their linear  combinations
are derivations, which are called { \it inner derivations}. The set ad$(B)$ of all the inner derivations of $B$ is an ideal of the derivation algebra Der$(B)$, and
ad$(B)$ is called {\it the inner derivation algebra of $B$}.

 Let  $B$ be a 3-Lie algebra,  $V$ be a vector space and  $\rho: B \wedge B\rightarrow gl(V)$ be a linear mapping. If $\rho$ satisfies for all $x_1, x_2, x_3, x_4\in B,$
\begin{equation}\label{eq:mod1}[\rho(x_1, x_2), \rho(x_3, x_4)]=\rho([x_1, x_2, x_3]_B, x_4)+ \rho(x_3,[x_1, x_2, x_4]_B),
\end{equation}
\begin{equation}\label{eq:mod2}\rho([x_1, x_2, x_3]_B, x_4)= \rho(x_1, x_2)\rho(x_3, x_4)+ \rho(x_2, x_3)\rho(x_1, x_4)+ \rho(x_3, x_1) \rho(x_2, x_4),
\end{equation}
\\then $(V, \rho)$ is called  {\it a representation}\cite{Kasymov} of the 3-Lie algebra $B$ (or is simply called {\it a 3-Lie algebra $B$-module}).
If there does not exist non-trivial modules of $V$, then $V$ is called {\it an irreducible module} of $B$.

For example,  for any 3-Lie algebra $B$, $(B, \mbox{ad})$ is a 3-Lie algebra  $B$-module, which is called the adjoint module of $B$,  where  ad$: B\wedge B\rightarrow gl(B)$,
for all $x_1, x_2\in B$, ad$(x_1, x_2)$ is the left multiplication.

 {\it A Cartan subalgebra} of a 3-Lie algebra $B$ is a nilpotent subalgebra of $B$ such that  $\forall y\in B$, if $y$ satisfies $[y, H, B]\subseteq H$, then $y\in H$.

 Let   $(V, \rho)$ be a 3-Lie algebra $B$-module, $\lambda\in  (H\wedge H)^*$. If
 \[
    V_{\lambda} = \{v\in V\mid \rho(h_1, h_2)v = \lambda(h_1, h_2) v, \quad \forall h_1, h_2\in H\}\neq 0,
\]
 then $\lambda$ is called a weight associated with the Cartan subalgebra $H$, and  $V_{\lambda}$ is called a weight space of $\lambda$.

 A module $(V, \rho)$ of a $3$-Lie algebra $B$ is referred to as a {\it weight module} if it admits a weight space decomposition
 $$V=\sum\limits_{\lambda\in\mathbb (H\wedge H)^*} V_\lambda,$$ where $H$ is a Cartan subalgebra of $B$.

A module of $B$ is  called  {\it a Harish-Chandra module} if it is a weight module and every weight space is finite-dimensional. A Harish-Chandra module of $B$ is an {\it intermediate series module} if the dimension of every weight space is  equal to one.

  Let A be a commutative associative algebra over $\mathbb F$ with a basis $\{L_r, M_r| r\in \mathbb Z\}$, where
  \begin{equation}
L_rL_s= L_{r+s}, ~~ M_rM_s= M_{r+s},~~  L_rM_s= 0, ~~\forall r, s\in \mathbb Z.
\label{eq:basisoperator}
\end{equation}

  Define linear mappings $\omega, \delta: A\rightarrow A$ by
\begin{equation}\label{eq:delta}
\delta(L_r)= rL_r, ~~~ \delta(M_r)= rM_r, \forall r\in \mathbb Z,
\end{equation}
\begin{equation}\label{eq:delta}
\omega(L_r)=M_{-r}, ~~~\omega(M_r)=L_{-r}, \forall r\in \mathbb Z.
\end{equation}
Then $\delta$ is a derivation of $A$, $\omega$ is an involution and satisfy $$\delta\omega+\omega\delta=0.$$ Thanks to Theorem 3.3 in \cite{Bai6}, $A$ is a simple canonical Nambu 3-Lie algebra with the multiplication
$$[x, y, z]_{\omega}=\begin{vmatrix}
\omega(x) & \omega(y) &\omega(z) \\
x & y & z  \\
\delta(x) & \delta(y) & \delta(z) \\
\end{vmatrix}, ~~~ \forall x, y, z\in A,$$ that is,
\begin{equation}\label{eq:table}
\begin{cases}
[L_r, L_s, M_t]_{\omega}= (s-r)L_{r+s-t}, ~~~ r, s, t\in Z,\\
[L_r, M_s, M_t]_{\omega}= (t-s)M_{s+t-r}, ~~~ r, s, t\in Z,\\
\mbox{and others are zero},
\end{cases}
\end{equation}
\\
 which is denoted by $A_{\omega}^{\delta}$.

\section{Modules and induced modules of $ A_{\omega}^{\delta}$}

In \cite{Bai6}, authors discussed the structure of the infinite dimensional simple 3-Lie algebra $A_{\omega}^{\delta}$. It is proved that
 the
adjoint representation $(A_{\omega}^{\delta}, \mbox{ad})$ is a Harish-Chandra module, and $A_{\omega}^{\delta}$ (as vector space) as the natural module of the inner derivation algebra ad$(A_{\omega}^{\delta})$ is an
intermediate series module.

In this section we construct infinite dimensional  3-Lie algebra $A_{\omega}^{\delta}$-module $T_{\lambda\mu}$, and discuss the relation between 3-Lie algebra $A_{\omega}^{\delta}$-modules and induced modules of inner derivation algebra ad$(A_{\omega}^{\delta})$.

In the following of the paper, suppose
\begin{equation}\label{eq:V}
V=\sum\limits_{\alpha\in \mathbb F}\mathbb F v_{\alpha},
\end{equation} is a vector space over  $\mathbb F$ with a basis $\{ v_{\alpha} ~~ | ~~ \alpha\in \mathbb  F\}$,  and for $\forall \alpha\in \mathbb F$, denotes~
\begin{equation}\label{eq:U} U(\alpha)= \sum\limits_{m\in \mathbb Z}\mathbb F v_{\alpha+ m}
\end{equation} the subspace of $V$.

\vspace{5mm}\subsection{3-Lie algebra $ A_{\omega}^{\delta}$-module $T_{\lambda\mu}$}

For some $\lambda, \mu\in \mathbb F$, define the linear mapping $\rho_{\lambda\mu}:  A_{\omega}^{\delta} \wedge  A_{\omega}^{\delta}\rightarrow gl (V)$  by
\begin{equation}
 \begin{cases}
 \rho_{\lambda\mu}(L_r, M_s)v_\alpha= (\lambda+\alpha+(s-r)\mu)v_{\alpha+s-r}, ~~~~r, s\in  \mathbb Z, \alpha\in\mathbb F,
\\
\rho_{\lambda\mu}(L_r, L_s) v_\alpha~~=~~0,\quad ~r, ~~s\in \mathbb Z,  \alpha\in\mathbb F,\\
\rho_{\lambda\mu}(M_r, M_s)v_\alpha=  0, \quad ~r, ~~s\in\mathbb Z,  \alpha\in\mathbb F. \\
\end{cases}\\
\label{eq:one}
\end{equation}

The pair $(V, \rho_{\lambda\mu})$ is  denoted by $T_{\lambda\mu}$.

\begin{theorem}\label{thm:T1} By the above notations, $T_{\lambda\mu}$ is a 3-Lie algebra $ A_{\omega}^{\delta}$-module if and only if $\mu=0$ or $\mu=1$.
And in the case  $\mu=0$ or $\mu=1$,

1) $T_{\lambda\mu}$ is a reducible module with submodules  $ U(\alpha)$~ for all $\alpha\in \mathbb F$;

2) if $\mu=0$ ,~then $\mathbb F v_{-\lambda}$~is a trivial submodule of $U(-\lambda)$;

3) if~$\mu =1$, $\alpha+ \lambda\in Z$, $\alpha\in \mathbb F$, then $\overline{U_\alpha}= \sum\limits_{m\in Z_{\neq -\lambda-\alpha}}\mathbb F v_{\alpha+ m}$ is an irreducible submodule of $U(\alpha)$;

4) for $\alpha\in \mathbb F$, $ U(\alpha)$~ is  irreducible  if and only if $\alpha+ \lambda\notin Z$.

\end{theorem}
\begin{proof} $\bullet$ First, we  prove that  $T_{\lambda\mu}$ is an $A_{\omega}^{\delta}$-module if and only if $\mu=0$ or $1$.

For $\forall r, m, s, l\in Z, \alpha, \lambda\in\mathbb F$, and  $\mu=0$  or $1$, thanks to  Eqs \eqref{eq:table} and \eqref{eq:one}, and a direct computation,
$$ [\rho_{\lambda\mu} (L_r, M_m), \rho_{\lambda\mu} (L_s, M_l)]v_\alpha
= \rho_{\lambda\mu} (L_r, M_m)\rho_{\lambda\mu} (L_s, M_l)v_\alpha-\rho_{\lambda\mu} (L_s, M_l)\rho_{\lambda\mu} (L_r, M_m)v_\alpha$$

\vspace{3mm}$\hspace{5.3cm}= (\lambda+\alpha+(l-s+m-r)\mu)(l-s-m+r)v_{\alpha-s+l-r+m},$

$$\rho_{\lambda\mu} ([L_r, M_m, L_s], M_l)+ \rho_{\lambda\mu} (L_s, [L_r, M_m, M_l])v_\alpha$$
$$=(r-s+l-m)(\lambda+\alpha+(l-r-s+m)\mu)v_{\alpha-s+m+l-r}.$$

Therefore,
\begin{eqnarray*}
[\rho_{\lambda\mu} (L_r, M_m), \rho_{\mu\nu} (L_s, M_l)]= \rho_{\mu\nu} ([L_r, M_m, L_s], M_l)+ \rho_{\mu\nu} (L_s, [L_r, M_m, M_l]).
\end{eqnarray*}
It is clear that

\begin{eqnarray*}
[\rho_{\lambda\mu} (L_r, L_m), \rho_{\mu\nu} (L_s, L_l)]= \rho_{\lambda\mu} ([L_r, L_m, L_s], L_l)+ \rho_{\lambda\mu} (L_s, [L_r, L_m, L_l])=0,
\end{eqnarray*}
\begin{eqnarray*}
[\rho_{\lambda\mu} (M_r, M_m), \rho_{\lambda\mu} (M_s, M_l)]= \rho_{\lambda\mu}([M_r, M_m, M_s], M_l)+ \rho_{\lambda\mu} (M_s, [M_r, M_m, M_l])=0.
\end{eqnarray*}
Therefore, Eq \eqref{eq:mod1} holds.

By Eqs \eqref{eq:table} and \eqref{eq:one}, for   $\forall m, l, r, s\in Z, \alpha\in \mathbb F$,
$$
\rho_{\lambda\mu}([L_r, L_s, M_m], M_l)v_\alpha =(s-r)(\lambda+\alpha+(l-r-s+m)\mu)v_{\alpha+l-r-s+m},
$$

$$
(\rho_{\lambda\mu} (L_r, L_s)\rho_{\lambda\mu} (M_m, M_l)+ \rho_{\lambda\mu} (L_s, M_m)\rho_{\lambda\mu} (L_r, M_l)+ \rho_{\lambda\mu} (M_m, L_r)\rho_{\lambda\mu} (L_s, M_l))v_\alpha$$
$$=
(\lambda+\alpha+(l-r)\mu)(\lambda+\alpha+l-r+(m-s)\mu)v_{\alpha+l-r+m-s}
$$
$$
-(\lambda+\alpha+(l-s)\mu)(\lambda+\alpha+l-s+(m-r)\mu)v_{\alpha+l-r+m-s}
$$
$$
=(s-r)(\lambda+\alpha+(l-r-s+m)\mu+(l-m)\mu(1-\mu))v_{\alpha+l-r-s+m};$$

$$\rho_{\lambda\mu}([M_m, M_l, L_r], L_s)=(m-l)(\lambda+\alpha+(m+l-r-s)\mu)v_{\alpha+m+l-r-s},$$

$$
\rho_{\lambda\mu} (M_l, L_r)\rho_{\lambda\mu} (M_m, L_s)+ \rho_{\lambda\mu} (L_r, M_m)\rho_{\lambda\mu} (M_l, L_s)+ \rho_{\lambda\mu} (M_m, M_l)\rho_{\lambda\mu} (L_r, L_s))v_\alpha$$
$$=(m-l)(\lambda+\alpha+(m+l-2s)\mu+(s-r)\mu^2)v_{\alpha+m+l-r-s}.
$$

From the above discussion, $\forall m, l, r, s\in Z, \alpha\in \mathbb F$, the identities

\vspace{3mm}\hspace{5mm}$
\rho_{\lambda\mu}([L_r, L_s, M_m], M_l)= \rho_{\lambda\mu} (L_r, L_s)\rho_{\lambda\mu} (M_m, M_l)+ \rho_{\lambda\mu} (L_s, M_m)\rho_{\lambda\mu} (L_r, M_l)$

\vspace{3mm}\hspace{4.4cm}$
+ \rho_{\lambda\mu} (M_m, L_r)\rho_{\lambda\mu} (L_s, M_l)
$
\quad and

\vspace{3mm}\hspace{5mm}$
\rho_{\lambda\mu}([M_m, M_l, L_r], L_s)=\rho_{\lambda\mu} (M_m, M_l)\rho_{\lambda\mu} (L_r, L_s)+ \rho_{\lambda\mu} (M_l, L_r)\rho_{\lambda\mu} (M_m, L_s)$

\vspace{3mm}\hspace{4.4cm}$ + \rho_{\lambda\mu} (L_r, M_m)\rho_{\lambda\mu} (M_l, L_s)
$

\vspace{3mm}\noindent hold if and only if $\mu=0$ or $\mu=1$, that is,
 \eqref{eq:mod2} holds if and only  $\mu=0$ or $\mu=1$.

 Therefore,  $T_{\lambda\mu}$ is an $A_{\omega}^{\delta}$-module if and only if $\mu=0$ or $\mu=1$.

\vspace{2mm}$\bullet\bullet$   Thanks to  \eqref{eq:one}, for $\forall \alpha\in \mathbb F$,  $U(\alpha)= \sum\limits_{m\in \mathbb Z}\mathbb F v_{\alpha+ m}$ is a submodule of $T_{\lambda\mu}$, we get 1).

If $\mu= 0$, then
\begin{equation}\label{eq:mu0}
\rho_{\lambda0}(L_r, M_s)v_{-\lambda}=0, \forall r, s\in Z,
\end{equation}~
it follows that $ \mathbb F v_{-\lambda}$~is a trivial submodule. It follows 2).

If~$\mu= 1$, then for all $ v_{\alpha+m}, v_{\alpha+ m'}\in U(\alpha)$,  $ m\neq m'\in Z$, there are distinct $r, s\in Z$, such that~$m'=m+s-r$. By \eqref{eq:one}
\begin{equation}\label{eq:m'}\rho_{\lambda1}(L_r, M_s)v_{\alpha+ m}= (\lambda+\alpha+m') v_{\alpha+m'}.
\end{equation}
Then in the case ~$\alpha+\lambda\in Z$, for $\forall m'\neq -\lambda-\alpha$, $\lambda+\alpha+m'\neq 0$. Thanks to \eqref{eq:m'}, $\overline{U_\alpha}= \sum\limits_{m\in Z_{\neq -\lambda-\alpha}}\mathbb F v_{\alpha+ m}$ is
an irreducible submodule of $U(\alpha)$. We get 3).

\vspace{2mm}$\bullet\bullet\bullet$   For $\forall \alpha\in \mathbb F_{\neq -\lambda}$,  $\mu=0,$ and
~$ v_{\alpha+m}, v_{\alpha+ m'}\in U(\alpha)$, then  $m\neq m'\in Z$,  and there are distinct  $r, s\in Z$ such that~$m'=m+s-r$. By \eqref{eq:one}
$$\rho_{\lambda 0}(L_r, M_s)v_{\alpha+ m}= (\lambda+\alpha+m) v_{\alpha+m'}.$$

If $\alpha+\lambda\notin Z$, then $\lambda+\alpha+m\neq 0$,~and $U(\alpha)=\rho_{\lambda0}(A_p, A_p)v_{\alpha+m}$,
it follows that ~$U(\alpha)$~is irreducible.

If ~$\alpha+\lambda\in Z$, then there are distinct  $r, s\in Z$, such that $\alpha+\mu=r-s$. Thanks to \eqref{eq:one}
$$\rho_{\lambda0}(L_r, M_s)v_{\alpha}=(\alpha+\lambda)v_{-\lambda}.$$
Therefore, ~$v_{-\lambda}\in U(\alpha)$. Thanks to \eqref{eq:mu0}, $\mathbb F v_{-\lambda}$ is a trivial submodule of  $U(\alpha)$.

If $\mu=1$ and $\alpha+\lambda\notin Z$, then for $ \forall v_{\alpha+m}, v_{\alpha+ m'}\in U(\alpha)$, $ m\neq m'\in Z$, $\lambda+\alpha+m'\neq 0$,~
$$U(\alpha)=\rho_{\lambda 1}(A_p, A_p)v_{\alpha+m},$$
it follows that ~$U(\alpha)$~is irreducible.

Summarizing the above discussion, we get 4). The proof is complete.
\end{proof}

\begin{theorem} \label{thm:T11}
The module $T_{\lambda0}$ and   $T_{\lambda1}$ are
intermediate series modules of $A_{\omega}^{\delta}$.
\end{theorem}

\begin{proof}
By \eqref{eq:table},
~~$H=\mathbb F L_0+\mathbb F M_0$ ~~ is a Cartan subalgebra of  the 3-Lie algebra  $A_{\omega}^{\delta}$. Thanks to \eqref{eq:one}, for $\forall \alpha\in \mathbb F,$
$$
\rho_{\lambda\mu}(L_0, M_0) v_{\alpha}
=(\lambda+\alpha)v_{\alpha}.
$$

Therefore, for $\forall \alpha\in \mathbb F$, $V_{\lambda+\alpha}=\mathbb F v_{\alpha}$ is the $\lambda+\alpha$-weight space with dimension one, and $V=\sum\limits_{\alpha\in \mathbb F} V_{\lambda+\alpha}$. The result follows.
\end{proof}

\subsection {Induced modules of 3-Lie algebras}

Let $B$ be a 3-Lie algebra over $\mathbb F$, $(W, \rho)$ be a 3-Lie algebra  $B$-module, where
$$\rho: B\wedge B\rightarrow gl(W).$$
Thanks to \eqref{eq:mod1} and \eqref{eq:mod2},
$$L_{\rho}(W)=\{~~ \rho(x, y): W\rightarrow W~~ | ~~\forall x, y\in B~~\}$$
is a subalgebra of $gl(W)$.

For $x, y, x', y'\in B$, if ad$(x, y)$=ad$(x', y')$, that is, $$[x, y, z]=[x', y', z], ~~ \forall z\in B,$$
then, for all $s, t\in M$,
$$[\rho(x, y)-\rho(x', y'), \rho(s, t)]=\rho([x, y, s], t)-\rho([x', y', s], t)+\rho(s, [x, y, t])-\rho(s, [x', y', t])=0,$$
it follows that $\rho(x, y)-\rho(x', y')$ is in the center of the Lie algebra $L_{\rho}(W)$, that is,
$$\rho(x, y)-\rho(x', y')\in Z(L_{\rho}(W)).$$
Therefore, in the case  $Z(L_{\rho}(W))=0$, if   ad$(x, y)$=ad$(x', y')$, then  $\rho(x, y)$$=\rho(x', y')$.

Therefore, if  $Z(L_{\rho}(W))=0$, we can define $\mathbb F$-linear mapping

\begin{equation}\label{eq:bar}
\bar{\rho}: \mbox{ad}(B)\rightarrow gl(W),\quad
\bar{\rho}(\mbox{ad}(x, y))=\rho(x, y), ~~ \forall x, y\in B.
\end{equation}

We get the following result.

\begin{theorem}\label{thm:barrho}
Let $B$ be a 3-Lie algebra over $\mathbb F$, $(W, \rho)$ be a 3-Lie algebra $B$-module, and $Z(L_{\rho}(W))=0$. Then $(W, \bar{\rho})$ is a Lie algebra ad$(L)$-module, where
$\bar{\rho}$ is defined by \eqref{eq:bar}.
\end{theorem}

\begin{proof} Thanks to \eqref{eq:mod1} and \eqref{eq:mod2}, for $\forall x, y, x', y'\in B$,

\vspace{3mm}\hspace{5mm}$[\bar{\rho}(\mbox{ad}(x, y)), \bar{\rho}(\mbox{ad}(x', y'))]$

\vspace{3mm}$=\bar{\rho}(\mbox{ad}(x, y))\bar{\rho}(\mbox{ad}(x', y'))-\bar{\rho}(\mbox{ad}(x', y')) \bar{\rho}(\mbox{ad}(x, y))
$

\vspace{3mm}$
=\rho(x, y)\rho(x', y')-\rho(x', y') \rho(x, y)=[\rho(x, y),\rho(x', y')]$

\vspace{3mm}$
=\rho([x, y, x'], y')+\rho(x', [x, y, y'])$

\vspace{3mm}$=\bar{\rho}(\mbox{ad}([x, y, x'], y'))+\bar{\rho}(\mbox{ad}(x', [x, y, y']))=\bar{\rho}\big([\mbox{ad}(x, y), \mbox{ad}(x', y')]\big).
$

Therefore, $(W, \bar{\rho})$ is  the inner derivation algebra ad$(B)$-module.
\end{proof}

The  inner derivation algebra ad$(B)$-module $(W, \bar{\rho})$ is called {\it the induced module} associated with the 3-Lie algebra $B$-module $(W, \rho)$, where
$\bar{\rho}$ is defined by \eqref{eq:bar},  and $(W, \bar{\rho})$ is simply called  {\it the induced module}.

\begin{lemma}\cite{Bai6}\label{lemma:1} The inner derivation algebra  ad$(A_{\omega}^{\delta})$ of the 3-Lie algebra $A_{\omega}^{\delta}$ has a basis  $\{~~p_r, q_r, x_r, z_r| r\in Z~~\}$, where
 \begin{equation}\label{eq:pqxzr}
\begin{cases}
p_r=\frac{1}{2}(ad(L_0, M_{-r})+ad(L_r, M_0)),~~ 0\neq r\in Z,\\
 q_r=\frac{1}{r}(ad(L_0, M_{-r})-ad(L_r, M_0)), ~~ 0\neq r\in Z, \\
x_r=\frac{1}{r}ad(L_r, L_{0}),~~ ~~~~z_r=\frac{1}{-r}ad(M_{-r}, M_0), ~~ 0\neq r\in Z,\\
p_0=ad(L_0, M_{0}), ~~~~~~~ q_0=ad(L_0, M_{0})-ad(L_1, M_1),\\
x_0=\frac{1}{2}ad(L_1, L_{-1}),~~~ z_0=\frac{1}{2}ad(M_{1}, M_{-1});
\end{cases}
\end{equation}
and the multiplication of   ad$(A_{\omega}^{\delta})$  in the basis is as follows:
 \begin{equation}\label{eq:adA}
\begin{cases}
[p_r, p_s]= (r-s)p_{r+s}, ~~ [p_r, q_s]=-sq_{r+s},\\
[p_r, x_s]= -sx_{r+s}, ~~[p_r, z_s]=-sz_{r+s},\\
[q_r, x_s]= -2x_{r+s},~~ [q_r, z_s]= 2z_{r+s}, \\
 [z_r, x_s]= q_{r+s},\\
  \mbox{others are zero, where }~~ r, s\in Z.
\end{cases}
\end{equation}
\end{lemma}

\vspace{2mm} Let $V$ be the vector space defined by  \eqref{eq:V}, for arbitrary $\lambda, \mu\in \mathbb F$, define $\mathbb F$-linear mapping
$$\psi_{\lambda\mu}: \mbox{ad}(A_{\omega}^{\delta}) \rightarrow gl(V),$$

\begin{equation}\label{eq:bar1}
\begin{cases}
\psi_{\lambda\mu}(p_r)v_{\alpha}=(\lambda+\alpha-r\mu)v_{\alpha-r}, ~~ \forall r\in Z, ~~ \alpha\in \mathbb F,\\
\psi_{\lambda\mu}(q_r)v_{\alpha}=\psi_{\lambda\mu}(x_r)v_{\alpha}=\psi_{\lambda\mu}(z_r)v_{\alpha}=0, ~~ \forall r\in Z, ~~ \alpha\in \mathbb F.\\
\end{cases}
\end{equation}

\begin{theorem}\label{thm:bar1} For  $\forall \lambda, \mu\in \mathbb F$, $(V, \psi_{\lambda\mu})$ is  a reducible $ \mbox{ad}(A_{\omega}^{\delta})$-module with submodules
 $ U(\alpha)=\sum\limits_{m\in Z} \mathbb Fv_{\alpha+m}$, $\forall \alpha\in \mathbb F$, and

1) if $\mu=0$, then $W=\mathbb F v_{-\lambda}$ is a trivial submodule; if ~$\mu=1$, then  $U'(-\lambda)= \sum\limits_{m\in Z_{\neq 0}}\mathbb F v_{-\lambda+ m}$~is irreducible;

2) $U(-\lambda)= \sum\limits_{m\in Z}\mathbb F v_{-\lambda+ m}$~is irreducible if and only if $\mu\neq 1$ and  $\mu \neq 0$.
\end{theorem}

\begin{proof}
By Eqs \eqref{eq:adA} and  \eqref{eq:bar1}, for  $\forall r, s\in Z$, $\lambda, \mu, \alpha\in \mathbb F$,
$$\psi_{\lambda\mu}([p_r, p_s])v_\alpha =(r-s)\psi_{\lambda\mu}(p_{r+s})v_{\alpha}=(r-s)(\lambda+\alpha-(r+s)\mu)v_{\alpha-r-s}.$$

\begin{eqnarray*}
& &[\psi_{\lambda\mu}(p_r),\psi_{\lambda\mu}( p_s)]v_\alpha\\
&=&(\psi_{\lambda\mu}(p_r)\psi_{\lambda\mu}( p_s)-\psi_{\lambda\mu}(p_s)\psi_{\lambda\mu}( p_r))v_\alpha\\
&=&(\lambda+\alpha-s\mu)[\lambda+\alpha-(s+r\mu)]v_{\alpha-s-r}-(\lambda+\alpha-r\mu)[\lambda+\alpha-(r+s\mu)]v_{\alpha-r-s}\\
&=&(r-s)(\lambda+\alpha-(r+s)\mu)v_{\alpha-r-s}.
\end{eqnarray*}

Then we get  $$\psi_{\lambda\mu}([p_r, p_s])= [\psi_{\lambda\mu}(p_r),\psi_{\mu\nu}( p_s)], ~~ \forall r, s\in Z, \lambda, \mu\in \mathbb F.$$
It follows that  $(V, \psi_{\lambda\mu})$ is an ad$(A_{\omega}^{\delta})$-module.

By Eqs \eqref{eq:adA} and  \eqref{eq:bar1}, and the complete  similar discussion to Theorem \ref{thm:T1},  $(V, \psi_{\mu\nu})$ satisfies  1) and 2).
\end{proof}

\begin{theorem}\label{thm:bar2}  In the case $\mu=0$ and $\mu=1$, the Lie algebra ad$(A_{\omega}^{\delta})$-module $(V, \psi_{\lambda\mu})$ is the induced module  associated with  $T_{\lambda\mu}$.
\end{theorem}

\begin{proof} From  Theorem \ref{thm:T1} and Eq \eqref{eq:one}, the center $Z(L_{\rho_{\lambda\mu}}(V))=0$.  Thanks to Theorem \ref{thm:barrho},  ad$(A_{\omega}^{\delta})$-module $(V, \overline{\rho_{\lambda\mu}})$ is the induced module of $T_{\lambda\mu}$, where $\overline{\rho_{\lambda\mu}}$ is defined by \eqref{eq:bar}. By   \eqref{eq:one} ,  $\forall r\in Z, \alpha, \lambda,  \mu\in \mathbb F$,

$$\overline{\rho_{\lambda\mu}}(\mbox{ad}(L_0, M_{-r}))v_\alpha=\rho_{\lambda\mu}(L_0, M_{-r})v_\alpha=(\lambda+\alpha-r\mu)v_{\alpha-r},$$

$$\overline{\rho_{\lambda\mu}}(\mbox{ad}(L_r, M_0))v_\alpha=\rho_{\lambda\mu}(L_r, M_{0})v_\alpha=(\lambda+\alpha-r\mu)v_{\alpha-r}.$$

Thanks to  \eqref{eq:pqxzr} and \eqref{eq:adA},

$$\overline{\rho_{\lambda\mu}}(p_r)v_\alpha=\overline{\rho_{\lambda\mu}}\Big(\frac{1}{2}\mbox{ad}(L_0, M_{-r})+\frac{1}{2}\mbox{ad}(L_r, M_0)\Big)v_\alpha=(\lambda+\alpha-r\mu)v_{\alpha-r}=\psi_{\lambda\mu}(p_r)v_{\alpha},$$

$$\overline{\rho_{\lambda\mu}}(q_r)v_\alpha=\psi_{\lambda\mu}(q_r)(v_{\alpha})=0,  \overline{\rho_{\lambda\mu}}(x_r)v_\alpha=\psi_{\lambda\mu}(x_r)v_{\alpha}=0,$$

$$\overline{\rho_{\lambda\mu}}(z_r)v_\alpha=\psi_{\lambda\mu}(z_r)(v_{\alpha})=0, ~~ \forall \alpha\in \mathbb F, ~~ r\in Z.$$

Follows from \eqref{eq:bar1}, we get
$$\overline{\rho_{\lambda\mu}}=\psi_{\lambda\mu},  ~~ \mbox{for}~~ \mu=0, ~~\mu=1.$$ It follows the result.
\end{proof}

\begin{remark}

From Theorem \ref{thm:bar2}, in the case $\mu=0, 1$, the inner derivation algebra ad$(A_{\omega}^{\delta})$-module $(V, \psi_{\lambda\mu})$ is the induced module.
But in the case  $\mu\neq 0, 1$, by Theorem \ref{thm:T1} and Theorem \ref{thm:bar2},  the inner derivation algebra ad$(A_{\omega}^{\delta})$-module $(V, \psi_{\lambda\mu})$ can not be induced by  $A_{\omega}^{\delta}$-modules.
\end{remark}

At last of the paper, we construct an ad$(A_{\omega}^{\delta})$-module $(V, \phi_{\mu})$, which is not an induced module.

Define  $\phi_{\mu}: \mbox{ad}(A_{\omega}^{\delta}) \rightarrow gl(V)$:
\begin{equation}\label{eq:bar3}
\begin{cases}
\phi_{\mu}(p_r)v_{\alpha}=(\alpha-r)v_{\alpha-r}, ~~ \forall r\in  Z_{\neq 0}, ~~ \alpha\in \mathbb F_{\neq 0},\\
\phi_{\mu}(p_r)v_{0}=r(\mu-r)v_{-r}, ~~ \forall r\in  Z_{\neq 0}, \\
\phi_{\mu}(p_0)v_{\alpha}=\alpha v_{\alpha}, ~~ \forall \alpha\in \mathbb F, \\
\phi_{\mu}(q_r)v_{\alpha}=\phi_{\mu}(x_r)v_{\alpha}=\phi_{\mu}(z_r)v_{\alpha}=0, ~~ \forall r\in  Z, ~~\forall  \alpha\in \mathbb F,\\
\end{cases}
\end{equation}

\begin{theorem}\label{thm:bar3}  For any $\mu\in \mathbb F,$ $(V, \phi_{\mu})$ is a reducible ad$(A_{\omega}^{\delta})$-module, where $\phi_{\mu}$ is defined by \eqref{eq:bar3},  and

1) for~$\forall \alpha\in \mathbb{F}-Z, U(\alpha)$ is an irreducible submodule;

2) $U(0)$ is a indecomposable  submodule containing an irreducible submodule $\bar U=\sum\limits_{m\in Z_{\neq 0}} \mathbb{F}v_{m}$;

3) $(V, \phi_{\mu})$ is not an induced module.
\end{theorem}

\begin{proof} Thanks to
 \eqref{eq:adA} and \eqref{eq:bar3}, for  $\forall r, s\in Z$, $\mu\in \mathbb F$,

$$\phi_{\mu}([p_r, p_s])v_\alpha = (r-s)\phi_{\mu}(p_{r+s})v_\alpha\\
=\begin{cases}
(r-s)\alpha v_\alpha , r+s=0, \alpha\in \mathbb F;\\
(r-s)(\alpha-r-s)v_{\alpha-r-s}, r+s\neq 0, \alpha\in \mathbb F_{\neq 0};\\
(r-s)(r+s)(\mu-r-s)v_{-r-s}, r+s\neq 0, \alpha= 0.
\end{cases}
$$

$$\phi_{\mu}(p_r)\phi_{\mu}(p_s)v_\alpha = \begin{cases}
\alpha^2 v_\alpha, ~~~~~s=r=0;\\
\alpha(\alpha-r)v_{\alpha-r}, ~~~~~s=0, r\neq 0, \alpha \neq 0;\\
0, ~~~~~s=\alpha=0, r\neq 0;\\
(\alpha-s)^2v_{\alpha-s}, ~~~~~s\neq0, \alpha \neq 0, r=0;\\
(\alpha-s)(\alpha-r-s)v_{\alpha-r-s}, ~~~~~s\neq 0, r\neq 0, \alpha\neq 0, \alpha\neq s;\\
0, ~~~~~s\neq 0, r\neq 0, \alpha\neq 0, \alpha = s;\\
s^2(s-\mu)v_{-s}, ~~~~~s\neq 0, \alpha=r=0;\\
s(\mu-s)(-r-s)v_{-r-s}, ~~~~~s\neq 0, \alpha=0, r\neq 0.
\end{cases}
$$

$$\phi_{\mu}(p_s)\phi_{\mu}(p_r)v_\alpha = \begin{cases}
\alpha^2v_\alpha, ~~~~~s=r=0;\\
\alpha(\alpha-s)v_{\alpha-s}, ~~~~~r=0, s\neq 0, \alpha \neq 0;\\
0, ~~~~~r=\alpha=0, s\neq 0;\\
(\alpha-r)^2v_{\alpha-r}, ~~~~~r\neq0, \alpha \neq 0, s=0;\\
(\alpha-r)(\alpha-r-s)v_{\alpha-r-s}, ~~~~~r\neq 0, s\neq 0, \alpha\neq 0, \alpha\neq r;\\
0, ~~~~~r\neq 0, \alpha\neq 0, s\neq 0, \alpha = r;\\
r^2(r-\mu)v_{-r}, ~~~~~r\neq 0, \alpha=s=0;\\
r(\mu-r)(-r-s)v_{-r-s}, r\neq 0, \alpha=0, s\neq 0.
\end{cases}
$$

Therefore,
\begin{eqnarray*}
[\phi_{\mu}(p_r), \phi_{\mu}(p_s)]v_\alpha = (\phi_{\mu}(p_r)\phi_{\mu}(p_s)-\phi_{\mu}(p_s)\phi_{\mu}(p_r))v_\alpha\\
=\begin{cases}
0, s=r=0;\\
r(\alpha-r)v_{\alpha-r}, s=0, r\neq 0, \alpha\neq 0;\\
r^2(\mu-r)v_{-r}, s=\alpha=0, r\neq 0;\\
s(s-\alpha)v_{\alpha-s}, s\neq0, \alpha\neq 0, r=0;\\
s^2(s-\mu)v_{-s}, s\neq 0, \alpha=r=0;\\
(r-s)(r+s)(\mu-r-s)v_{-r-s}, r\neq 0, s\neq 0, \alpha=0;\\
(r-s)(\alpha-r-s)v_{\alpha-s-r}, s\neq 0, r\neq 0, \alpha\neq 0;
\end{cases}
\end{eqnarray*}
 and
\begin{eqnarray*}
\phi_{\mu}([p_r, p_s])v_\alpha &=&[\phi_{\mu}(p_r), \phi_{\mu}(p_s)]v_\alpha\\
&=& (\phi_{\mu}(p_r)\phi_{\mu}(p_s)-\phi_{\mu}(p_s)\phi_{\mu}(p_r))v_\alpha.
\end{eqnarray*}

It follows that $(V, \phi_{\mu})$ is an ad$(A_{\omega}^{\delta})$-module.

\vspace{2mm}For any ~$\alpha\in \mathbb F -Z$, and arbitrary
distinct $m, m'\in Z$, there are nonzero $r\in Z$ such that
~$m'=m-r$. By \eqref{eq:bar3}, we nave
\begin{eqnarray*}
\phi_{\mu}(p_r) v_{\alpha+ m}=(m'+\alpha)v_{\alpha+m'}.
\end{eqnarray*}
Thanks to ~$\alpha\in \mathbb F -Z$, $m'+\alpha\neq 0$, and $U(\alpha)=\phi_{\mu}\big($ad$(A_{\omega}^{\delta})\big)(v_{\alpha+m})$. We get 1).

If $\alpha\in Z$, then  $U(\alpha)=U(0)$.

For arbitrary distinct $m, m'\in Z_{\neq 0}$, there is ~nonzero $r_0\in Z$,~such that~$m'=m-r_0$. Thanks to \eqref{eq:bar3},
$$\phi_{\mu}(p_{r_0})v_{m}=m'v_{m'},$$
Therefore, $\bar U=\sum\limits_{m\in Z_{\neq 0}} \mathbb{F}v_{m}$ is an irreducible submodule.

Since for any nonzero $r\in Z$,
$$ \phi_{\mu}(p_r) v_{0}=r(\mu-r)v_{-r}\in U(0).$$
Follows from $\bar U=\sum\limits_{m\in Z_{\neq 0}} \mathbb{F}v_{m}$ is irreducible,  $U(0)$ is an indecomposable. It follows 2).

If $(V, \phi_{\mu})$ is an induced module, then by \eqref{eq:bar}, $\phi_{\mu}=\bar{ \varphi}$, where

$$ \varphi: A_{\omega}^{\delta}\wedge A_{\omega}^{\delta}\rightarrow gl(V), ~~\mbox{and}~~~~ \phi_{\mu}(\mbox{ad}(x, y))= \varphi(x, y),~~ \forall x, y\in A_{\omega}^{\delta},$$
and $(V, \varphi)$ is a 3-Lie algebra $A_{\omega}^{\delta}$-module.

Thanks to Eqs \eqref{eq:bar}, \eqref {eq:pqxzr} and \eqref{eq:bar3}, $\forall r, s\in Z$, $\alpha\in \mathbb F$, $\varphi$ satisfies that

\begin{equation}
 \begin{cases}
 \varphi(L_r, M_s)v_\alpha=\begin{cases}
\alpha v_\alpha, ~~~~ r= s, \alpha\in \mathbb{F},\\
(s-r+\alpha)v_{s-r+\alpha},~ r\neq s,  \alpha \neq 0,\\
(r-s)(s-r+\mu)v_{s-r}, ~~~~ r\neq s,  \alpha=0;\\
\end{cases}\\\\
 \varphi(L_r, L_s) v_\alpha=\varphi(M_r, M_s)v_\alpha=~~0, \quad ~r, ~~s\in Z,  \alpha\in \mathbb{F}. \\
\end{cases}\\
\label{eq:two}
\end{equation}

Therefore, $\forall r, m, s, l\in Z, \alpha, \mu\in\mathbb F$,
\begin{eqnarray*}
 \varphi([L_r, L_s, M_m], M_l)v_\alpha=\begin{cases}
(s-r)\alpha v_\alpha, ~~~~r+s=m+l, \alpha\in\mathbb{F},\\
(s-r)(m+l-r-s+\alpha)v_{m+l-r-s+\alpha},\\
{\mbox where } ~~r+s\neq m+l, \alpha\in\mathbb F_{\neq 0},\\
(s-r)(r+s-m-l)(l-r-s+m+\mu)v_{l-r-s+m}, \\
{\mbox where } ~~ r+s\neq m+l, \alpha=0;
\end{cases}
\end{eqnarray*}

\begin{eqnarray*}
 \varphi(L_s, M_m)\varphi(L_r, M_l)v_\alpha
=\begin{cases}
\alpha^2 v_\alpha, ~~~~s=m, r=l, \alpha\in\mathbb{F},\\
0, ~~~~r=l, s\neq m, \alpha=0;\\
\alpha(m-s+\alpha)v_{m-s+\alpha},~~s\neq m, r=l, \alpha\neq 0,\\
(l-r+\alpha)^2v_{l-r+\alpha}, ~~r\neq l, s=m, \alpha\neq 0,\\
-(l-r)^2(l-r+\mu)v_{l-r}, ~~r\neq l, s=m, \alpha=0,\\
0, ~~~~r\neq l, s\neq m, \alpha\neq 0, l-r+\alpha=0,\\
(l-r+\alpha)(m-s+l-r+\alpha)v_{m-r+l-s+\alpha}, \\
{\mbox where } ~~r\neq l, s\neq m, \alpha\neq 0, l-r+\alpha\neq 0,\\
(r-l)(l-r+\mu)(l-s+m-r)v_{l-s+m-r}, \\
{\mbox where } ~~r\neq l, s\neq m, \alpha= 0;
\end{cases}\end{eqnarray*}

\begin{eqnarray*}
 \varphi(M_m, L_r)\varphi(L_s, M_l)v_\alpha
=\begin{cases}
-\alpha^2 v_\alpha, ~~~~r=m, s=l, \alpha\in\mathbb{F},\\
0, ~~~~s=l, r\neq m, \alpha=0;\\
-\alpha(m-r+\alpha)v_{m-r+\alpha}, ~~~~s=l, r\neq m, \alpha\neq 0,\\
-(l-s+\alpha)^2v_{l-s+\alpha}, ~~~~s\neq l, r=m, \alpha\neq 0;\\
(l-s)^2(l-s+\mu)v_{l-s}, ~~~~s\neq l, r=m, \alpha=0,\\
0, ~~~~s\neq l, r\neq m, \alpha\neq 0, l-s+\alpha=0;\\
-(l-s+\alpha)(m-r+l-s+\alpha)v_{m-r+l-s+\alpha}, \\
{\mbox where } ~~s\neq l, r\neq m, \alpha\neq 0, l-s+\alpha\neq 0,\\
-(s-l)(l-s+\mu)(m-r+l-s)v_{m-r+l-s}, \\
{\mbox where } ~~s\neq l, r\neq m, \alpha= 0;
\end{cases}\end{eqnarray*}

\vspace{3mm}\hspace{1.5cm}$
 \varphi(L_r, L_s)\varphi (M_m, M_l)v_\alpha=0.
$

\vspace{3mm}Then for $r=4$, $s=3$, $m=2$, $l=1$, $\alpha=0$, and $\forall \mu\in \mathbb F,$
\begin{eqnarray*}
 \varphi([L_4, L_3, M_2], M_1)v_\alpha=-4(\mu-4)v_{-4},
\end{eqnarray*}
\begin{eqnarray*}
& & \varphi(L_3, M_2)\varphi(L_4, M_1)v_\alpha+ \varphi (L_4, L_3)\varphi (M_2, M_1)v_\alpha+\varphi(M_2, L_4)\varphi(L_3, M_1)v_\alpha\\
&=&(-12(\mu-3)+8(\mu-2))v_{-4}=-4(\mu-5)v_{-4},
\end{eqnarray*}
that is, for some $l, r, s, m\in Z,$ we have
\begin{eqnarray*}
& & \varphi([L_r, L_s, M_m], M_l)v_\alpha~~~~~~~~~~~~~~~~~~~~~~~~~~~~~~~~~~~~~~~~~~~~~~~~~~~~~~~~~~\\
&\neq&  \varphi(L_s, M_m)\varphi(L_r, M_l)v_\alpha+ \varphi(L_r, L_s)\varphi(M_m, M_l)v_\alpha+\varphi(M_m, L_r)\varphi(L_s, M_l)v_\alpha.
\end{eqnarray*}
Contradiction. It follows 3). The proof is complete. \end{proof}

\noindent
{\bf Acknowledgements. } The first author was supported in part by  the Natural
Science Foundation of Hebei Province (A2018201126).

\bibliography{}

\begin{thebibliography}{999999}
\bibitem{Aw} H. Awata, M. Li, D. Minic, et al. On the Quantization of Nambu Brackets, {\it  Journal of High Energy Physics,} 2001, 2001(2):69-82.
\bibitem{Gustavsson}  A. Gustavsson, Algebraic structures on parallel M2-branes, {\it Nuclear Phys. B.,} 2009, 811(1):66-76.
\bibitem{Bagger1}  J. Bagger, N. Lambert, Gauge symmetry and supersymmetry of multiple M2-branes, {\it Phys.rev.d.,}
2008, 77(6):215-240.
\bibitem{Bagger2}  J. Bagger, N. Lambert, Modeling Multiple M2＊s, {\it Phys.rev.d.,} 2006, 75.

\bibitem{Bai1} R. Bai, G. Song and Y. Zhang, On classification of n-Lie algebras. Front. Math. China 6
(2011), no. 4, 581每606.
\bibitem{Bai2} R. Bai, W. Wu and Z. Chen ㄛ Classifications of (n + k)-dimensional metric n-Lie algebras, {\it J. Phys. A: Math. Theor.} 46 (2013) 145202.
\bibitem{Bai3} R. Bai, L. Zhang, Y. Wu and Z. Li, On 3-Lie algebras with abelian ideals and subalgebras, {\it Linear Algebra Appl.}, 2013, 438(5): 2072每2082.
\bibitem{Bai4} R. Bai, C. Bai and J. Wang, Realizations of 3-Lie algebras, {\it J. Math. Phys.,} 51 (2010), 063505
\bibitem{Bai5} R. Bai, Y. Wu, Constructions of 3-Lie algebras, {\it Linear Multilinear Algebra,} 2015, 63(11): 2171-2186.

\bibitem{Bai6} R. Bai, Z. Li and W. Wang, Infinite-dimensional 3-Lie algebras and their connections to Harish-Chandra
modules, {\it Front. Math. China,} 2017, 12.3: 515-530.
\bibitem{DebeSR} J. DeBellis, C. Saemann and R. Szabo, Quantized Nambu Poisson manifolds and n-Lie algebras, {\it J.
Math. Phys.} 2010, 51(12): 122303.
\bibitem{Filippov} V. Filippov, $n$-Lie algebras,  {\it Sib. Mat.Zh.,} 26(6) (1985): 126-140.
\bibitem{Farrill} J. Figueroa-O∩Farrill, Deformations of 3-algebras, {\it  J. Math. Phys.,} 50 (2009), no. 11, 113514, 27 pp.
\bibitem{Gautheron} P. Gautheron, Some remarks concerning Nambu mechanics. Lett. Math. Phys. 37 (1996)
103每116.
\bibitem{Izquier1} J. A. de Azc∩arraga and J. M. Izquierdo, n-ary algebras: a review with applications, {\it J. Phys.
A: Math. Theor.,} 43 (2010), 293001.
\bibitem{Izquier2} J. A. de Azc∩arraga and J. M. Izquierdo, Cohomology of Filippov algebras and an analogue
of Whitehead＊s lemma, {\it J. Phys. Conf. Ser.,} 175: 012001, (2009).
\bibitem{Kasymov}Sh. M. Kasymov, On a theory of n-Lie algebras. (Russian) {\it Algebra i Logika} 1987, 26(3): 277每297.
\bibitem{Ling} Ling W., On the structure of n-Lie algebras, Ph.D. thesis, University-GHS-Siegen, 1993.
\bibitem{Nambu} Nambu Y. Generalized Hamiltonian Dynamics, {\it Physical Review D Particles  Fields,} 1999, 7(8): 2405-2412.
\bibitem{Pozhidaev1} Pozhidaev A. P., Monomial n-Lie algebras, {\it~Algebra Logic}. 1998, 37(5): 307-322.
\bibitem{Pozhidaev2} Pozhidaev A. P., Simple n-Lie algebras, {\it~Algebra Logic}. 1999, 38(3): 181-192.
\bibitem{Shengyunhe1} Y. Sheng, R. Tang, Symplectic, product and complex structures on 3-Lie algebras, {\it Journal of Algebra, } 2018, 508: 256-300
\bibitem{Shengyunhe2}
J. Liu, A. Makhlouf, Y. Sheng
A new approach to representations of 3-Lie algebras
and abelian extensions, A new approach to representations of 3-Lie algebras
and abelian extensions,{\it Algebra Representation Theor.,} 2017,20:1415-1431.
\bibitem{T} L. Takhtajan, On foundation of the generalized Nambu mechanics, {\it Comm. Math. Phys.}, 1994, 160: 295 - 315.






































\end{thebibliography}

\end{document}